\documentclass{amsart}
\usepackage{amsmath,amssymb,amsfonts}
\usepackage{graphicx}
\usepackage{color}

\usepackage{amsthm}
\usepackage{amsrefs}
\usepackage{mathrsfs}
\usepackage{stmaryrd}
\usepackage[normalem]{ulem}

\newtheorem{theorem}{Theorem}[section]
\newtheorem{lemma}[theorem]{Lemma}

\newtheorem{proposition}[theorem]{Proposition}
\newtheorem{definition}[theorem]{Definition}

\theoremstyle{definition}
\newtheorem{example}[theorem]{Example}

\theoremstyle{remark}

\numberwithin{equation}{section}

\newcommand{\eps}{\varepsilon}

\def\beq{\begin{equation}}
\def\eeq{\end{equation}}

\def\Dio{{\rm Dio}}
\def\dio{{\rm dio}} 
\def\ice{{\rm ice}}
\def\rep{{\rm rep}}

\def\bfx{{\mathbf x}}
\def\bfz{{\mathbf z}}
\def\cG{{\mathcal G}}

\begin{document}

\title[irrationality exponent of numbers with low complexity expansion]{On the irrationality exponent of real numbers with low complexity expansion}

\author{Yann Bugeaud}
\address{I.R.M.A., UMR 7501, Universit\'e de Strasbourg et CNRS, 7 rue Ren\'e Descartes, 67084 Strasbourg Cedex, France}
\address{Institut universitaire de France}
\email{bugeaud@math.unistra.fr}

\author{Hajime Kaneko}
\address{Institute of Mathematics, University of Tsukuba, 1-1-1 Tenodai, Tsukuba, Ibaraki, 305-8571, Japan}
\email{kanekoha@math.tsukuba.ac.jp}

\author{Dong Han Kim}
\address{Department of Mathematics Education,
Dongguk University, 30 Pildong-ro 1-gil, Jung-gu, Seoul 04620, Korea.}
\email{kim2010@dgu.ac.kr}

\begin{abstract}
Let $\xi$ be a real number and $b \ge 2$ an integer.
We study the relationship between the irrationality exponent of $\xi$ and the subword complexity $p(n, \bfx)$ of the $b$-ary expansion $\bfx$ of $\xi$, where 
$p(n, \bfx)$ counts the number of distinct blocks of length $n$ in $\bfx$, for $n \ge 1$. 
If the irrationality exponent of $\xi$ is equal to $2$, 
which is the case for almost all real numbers $\xi$, 
we show that the limit superior of the sequence $(p(n, \bfx) / n)_{n \ge 1}$  
is at least equal to 4/3. 
The proof is based on a careful study of the evolution of the 
Rauzy graphs of infinite words of low complexity. 
\end{abstract}




\def\Dio{{\rm Dio}}
\def\dio{{\rm dio}} 
\def\ice{{\rm ice}}
\def\rep{{\rm rep}}
\def\Card{{\rm Card}}

\def\bfa{{\bf a}}

\maketitle

\section{Introduction}

Let $\xi$ be an irrational real number.
Its irrationality exponent $\mu(\xi)$ is the supremum of the real numbers $\mu$ for which the inequality
$$
\left| \xi - \frac pq \right| < \frac{1}{q^\mu} 
$$
has infinitely many rational solutions $p/q$ with $q \ge 1$. 
It follows from the theory of continued fractions that $\mu(\xi)$ is always at least equal to $2$ and an easy 
covering argument shows that equality holds for almost all $\xi$ with respect to the Lebesgue measure.
Recently, Bugeaud and Kim \cite{BuKi17}  established a
connection between the irrationality exponent of a 
real number and the complexity of its expansion in a given integer base.  
Before stating their result, let us introduce some notions from combinatorics on words. 

Let $\mathcal A$ be a finite set called an alphabet and denote by $|\mathcal A|$ its cardinality.
A word over $\mathcal A$ is a finite or infinite sequence of elements of $\mathcal A$.
For a (finite or infinite) word ${\bf x} = x_1 x_2 \ldots$ written over $\mathcal A$,
let $n \mapsto p (n, {\bf x})$ denote its subword complexity function of $\bf x$ which counts the number of different subwords of length $n$ occurring in $\mathbf x$, that is,
$$
p (n,{\bf x}) = {\rm Card} \{ x_{j+1} x_{j+2} \dots x_{j+n} : j \ge 0 \}, \quad n \ge 1.
$$
Clearly, we have
$$
1 \le p(n, {\bf x}) \le |\mathcal A|^n, \quad n \ge 1.
$$
We set $p(0,\bfx)=1$ for convenience.
The Morse-Hedlund Theorem \cite{MoHe40} states that 
if ${\bf x}$ is ultimately periodic, then there exists an integer $C$ such that $p(n, {\bf x}) \le C$ for $n \ge 0$; 
otherwise, we have
\begin{equation}\label{eq:MH}
p(n+1, {\bf x}) \ge p(n, {\bf x}) + 1, \quad n \ge 0,  
\end{equation}
thus $p(n, {\bf x}) \ge n+1$ for $n \ge 0$. 
There exist uncountably many infinite words 
${\bf s}$ over $\{0, 1\}$ such that $p(n, {\bf s}) = n+1$ for $n \ge 0$. 
These words are called Sturmian words. 
Classical references on combinatorics on words and
on Sturmian sequences include \cites{Fogg02,Loth02,AlSh03}. 
The words $\bfx$ for which there exists an integer $C$ such that $p(n, {\bf x}) \le n + C$ for $n \ge 0$ are called quasi-Sturmian. 
See \cite{Cass98} for more details.

Let $\xi$ be a real number. 
Let $b$ be an integer greater than or equal to $2$. 
There exists a unique infinite sequence 
$(x_j)_{j \ge 1}$ of integers from $\{0, 1, \ldots, b-1\}$, called the $b$-ary expansion of $\xi$, such that
\begin{equation}  \label{expbase}
\xi = \lfloor \xi \rfloor + \sum_{j\ge 1} \, \frac{x_j}{b^j}    
\end{equation}
and $x_j \not= b-1$ for infinitely many indices $j$. 
Here, $\lfloor \cdot \rfloor$ denotes the integer part function. 
The sequence $(x_j)_{j \ge 1}$ is ultimately periodic if, and only if, $\xi$ is rational. 
A natural way to measure the
complexity of the real number $\xi$ written in base $b$ 
as in \eqref{expbase} is to count the number of
distinct blocks of given length in the infinite word ${\bf x} = x_1 x_2 \ldots$. 
We set
$$
p(n, \xi, b) = p (n, \bfx), \quad n \ge 1,
$$
and call the function $n \mapsto p(n, \xi, b)$ the complexity function of $\xi$ in base $b$. 

Adamczewski \cite{Adam10} has observed that the statements on the combinatorial structure of 
Sturmian words established in \cite{BeHoZa06} and \cite{AdBu11} almost immediately imply that the irrationality exponent of any real number, whose sequence of digits in some integer base is quasi-Sturmian, is greater than $2$. 
The main result of \cite{BuKi17} asserts that, if the complexity function of $\xi$ in base $b$ grows sufficiently slowly, then the sequence of digits of $\xi$ contains very large 
repetitions, that causes the irrationality exponent of $\xi$ to exceed $2$. 

\begin{theorem}[\cite{BuKi17}] \label{Pisa} 
Let $b \ge 2$ be an integer and $\xi$ an irrational real number.
If $\mu$ denotes the irrationality exponent of $\xi$, then
$$
\liminf_{n \to + \infty} \, \frac{p(n, \xi, b)}{n} \ge 
1 +  \frac{1 - 2 \mu (\mu - 1) (\mu - 2)}{\mu^3 (\mu - 1)} 
$$
and
$$
\limsup_{n \to + \infty} \, \frac{p(n, \xi, b)}{n}  
\ge  1  +  \frac{1 - 2 \mu (\mu - 1) (\mu - 2)}{3 \mu^3 - 6 \mu^2 + 4 \mu - 1}.  
$$
\end{theorem}

The lower bound for the limsup obtained in Theorem \ref{Pisa} 
has been refined in \cite{BuKi25} as follows. 

\begin{theorem} \label{PisaBis}
Let $b \ge 2$ be an integer and $\xi$ an irrational real number.
If $\mu$ denotes the irrationality exponent of $\xi$, then 
$$
\liminf_{n \to + \infty} \, \frac{p(n, \xi, b)}{n} \ge 
1+ \frac{-\mu^3 + 2 \mu^2 + \mu - 1}{\mu^4 - 2 \mu^3 + 3 \mu^2 - 3 \mu + 1}
$$
and
\begin{equation} \label{lbPisaBis}
\limsup_{n \to + \infty} \, \frac{p(n, \xi, b)}{n}  
\ge  \frac{\mu + \sqrt{ 4(\mu - 1)^3 + \mu^2 } }{2 \mu (\mu - 1)}.
\end{equation}
\end{theorem}

The purpose of the present paper is, 
by means of a totally different approach, 
to improve \eqref{lbPisaBis}.

\begin{theorem}\label{thm1}
Let $b \ge 2$ be an integer and $\xi$ an irrational real number.
Let $\mu$ denote the irrationality exponent of $\xi$.
If $\mu = 2$, then 
\begin{equation} \label{bestminorp}
\limsup_{n \to  + \infty} \frac{p(n, \xi, b)}{n} \ge \frac{4}{3}. 
\end{equation}
If $\mu > 2$, then 
\begin{equation} \label{bestminor}
\limsup_{n \to  + \infty} \frac{p(n, \xi, b)}{n} \ge 1 + 
\frac{\mu + 1 - \sqrt{81 (\mu-2)^2 - 10 \mu + 29}}{8(\mu - 2)}. 
\end{equation}
\end{theorem}

Note that the right-hand side of \eqref{bestminor} tends to $4/3$ when 
$\mu$ tends to $2$. Furthermore, it is greater than one (that is, it gives a 
non-trivial result) for $\mu < 2.2$. 
Observe also that Theorems \ref{Pisa} and \ref{PisaBis} give respectively the 
lower bounds 
$8/7$ and $(1 + \sqrt{2})/2 = 1.207 \ldots$ for 
the limit superior when $\mu = 2$. 
Thus, Theorem \ref{thm1} is noticeably stronger. 
As explained in \cite{BuKi25}, the method developed in \cite{BuKi17} 
and slightly refined in \cite{BuKi25} 
is unlikely to yield the result of the present paper.

Here, we obtain Theorem \ref{thm1} after a very detailled study of Rauzy graphs, introduced
by G\'erard Rauzy in 1983 \cite{Rauz83}. 
While this approach gives 
better estimates for the limsup, it gives unfortunately no estimate for the liminf, 
unlike the approach followed in \cites{BuKi17,BuKi25}. 

Theorem~\ref{thm1} is a restatement of Theorem~\ref{thm2} below, which involves 
the repetition function $n \mapsto r(n, \bfx)$ of an infinite word $\bfx$, defined 
for any positive integer $n$ by
\begin{align*}
r(n,\mathbf x) &= \max \{ m \ge 1 \, | \, x_1 \dots x_n, x_2 \dots x_{n+1}, \dots, x_{m} \dots x_{n+m-1} \text{ are all distinct} \} \\
&= \min \{ m \ge 1 \, | \, x_{m+1} \dots x_{m+n} = x_{i+1} \dots x_{i+n} \text{ for some } 0 \le i < m \}.
\end{align*}
Observe that the present definition of $r(n,\mathbf x)$ differs by $-n$ from that 
given in \cites{BuKi17,BuKi19,BuKi25} and that we have $p(n, \bfx) \ge r(n, \bfx)$ for $n \ge 1$. 

Theorem~\ref{thm2} asserts that if $p(n, \bfx)$ is sufficiently small for 
every large value of $n$, then there are some positive $\eps$ and 
arbitrarily large values of $n$ such that the prefix of length 
$\lfloor (2 - \eps) n \rfloor$ of $\bfx$ has two (necessarily overlapping) 
occurrences of a same block of length $n$. 

\begin{theorem}\label{thm2}
Let $b \ge 2$ be an integer. 
Let $\bfx = x_1 x_2 \ldots$ be an infinite word over $\{0, 1, \ldots , b-1\}$ which is not ultimately periodic and satisfies
$$
\rho = \limsup_{n \to  + \infty} \frac{p(n, \bfx)}{n} < \frac{4}{3}.
$$
Then, we have 
$$
\liminf_{n \to \infty} \frac{r(n,\mathbf x)}{n} \le  
 1 - 
\frac{4-3\rho}{2(1+2\rho)(2-\rho)}
$$
and
\begin{equation}\label{thm2_bound}
\mu \Bigl( \, \sum_{j \ge 1} \, \frac{x_j}{b^j} \, \Bigr) \ge 2 + 
\frac{4-3\rho}{\rho (9- 4\rho)}.
\end{equation}
\end{theorem}

We note that by setting $\mu = \mu \bigl( \, \sum_{j \ge 1} \, \frac{x_j}{b^j} \, \bigr)$
and solving 
\eqref{thm2_bound}, we obtain 
\eqref{bestminorp} and \eqref{bestminor}.

Let $\bfx$ denote the word given by the $b$-ary expansion 
of an irrational, real number $\xi$. 
In \cites{BuKi17,BuKi25}, we have shown that there are 
$\delta, \delta', \delta''> 0$ such that, 
if $r(n, \bfx) \le (1 + \delta) n$ for every sufficiently large integer $n$, then there 
are arbitrarily large $n$ such that $r(n, \bfx) \le (1 - \delta') n$
(see \cite{BuKi17}*{Theorem 3.5}, \cite{BuKi25}*{Proposition 1.6}), thus, by 
\cite{BuKi17}*{Lemma 3.6}, we have
$\mu(\xi) \ge 2 + \delta''$. Said differently, if $\mu(\xi) < 2 + \delta''$, then 
there are arbitrarily large $n$ such that 
$p(n, \bfx) \ge r(n, \bfx) \ge (1 + \delta) n$.  

In the present paper, our approach is different. 
Our assumption is on the function $n \mapsto p(n, \bfx)$ and not on the function $n \mapsto r(n, \bfx)$. We show that there are 
$\delta, \delta' > 0$ such that, 
if $p(n, \bfx) \le (1 + \delta) n$ for every sufficiently large integer $n$, then, there 
are arbitrarily large $n$ such that $r(n, \bfx) \le (1 - \delta') n$. 

To achieve this, we study very carefully 
the combinatorial structure of words of small complexity, 
focusing on the evolution of their Rauzy graphs defined in 
Subsection \ref{SubsecRauzyG}. 
Previous works on these words and their Rauzy graphs include \cites{Aber00,Aber03}, where Aberkane
studies, among other, the infinite words $\bfx$ such that $p(n, \bfx) / n$ tends to $1$ as 
$n$ tends to infinity. Also, Heinis \cite{Hein02} established that if the sequence 
$(p(n, \bfx) / n)_{n \ge 1}$ has a limit, then this limit cannot be in the open 
interval $(1, 2)$. This interval is the best one can obtain, since there exist 
sequences $\bfx$ with $p(n, \bfx) = 2n+1$ for $n \ge 1$. 
Certain such sequences, under the additional 
assumption of minimality, have been 
studied by Arnoux and Rauzy \cite{ArRa91}, again by means of the evolution of their Rauzy graphs. 
Recently, Creutz and Pavlov \cite{CrPa23} showed that if $X$ is an infinite transitive subshift with $\limsup p(n)/n < 4/3$, where $p(n)$ is the number of $n$-subwords appearing in sequences in the subshift $X$, then 
$X$ is uniquely ergodic and its unique ergodic measure has discrete spectrum.
See also \cite{CrPa}.

One important difficulty we had to overcome is that there is no additional 
assumption on $\bfx$ in Theorem \ref{thm2}. In particular, $\bfx$ is not supposed 
to be minimal, nor even recurrent: it may happen that arbitrarily long prefixes 
of $\bfx$ occur only once in $\bfx$.

\section{Auxiliary lemmas}

Throughout the paper, we assume that the words are written over a finite alphabet. 

\subsection{Special words and the Rauzy graph} \label{SubsecRauzyG}

Let $n$ be a positive integer. An $n$-word is a word of length $n$. If $\bfx$ is an infinite word, an $n$-subword 
(or, $n$-factor) of $\bfx$ is a factor of $\bfx$ of length $n$. 

\begin{definition}   \label{defspecial}
Let $\bfx$ be an infinite word and $n$ a positive integer.
We let  $\mathcal F_n(\mathbf x)$ denote the set of $n$-subwords of $\mathbf x$.
A word $w$ in $\mathcal F_n(\mathbf x)$ is called right-special if there exist 
distinct letters $x$, $y$ such that $wx$ and $wy$ are subwords of $\bfx$. 
A word $w$ in $\mathcal F_n(\mathbf x)$ is called left-special if there exist distinct letters $x$, $y$ such that $xw$ and $yw$ are subwords of $\bfx$. 
Any word in $\mathcal F_n(\mathbf x)$ which is right-special and left-special is called a bi-special word.
\end{definition}

The following lemma is an immediate consequence of Definition \ref{defspecial}. 

\begin{lemma}\label{special}
Let $\bfx$ be an infinite word and $m, n$ integers with $n > m \ge 1$.

(i) If $w$ is a left-special $n$-word of $\bfx$, then the prefix of $w$ of length $m$ is a left special $m$-word of $\bfx$. 

(ii) If $w$ is a right-special $n$-word of $\bfx$, then the suffix of $w$ of length $m$ is a right special $m$-word of $\bfx$. 
\end{lemma}

\begin{lemma}\label{nonep}
Suppose that $\bfx$ is a non-ultimately periodic word.
Then there is at least one left-special word $w$ with the property that there are two distinct letters $x, y$ such that the words $xw$ and $yw$ occur infinitely many times in $\bfx$.
\end{lemma}

\begin{proof}
If there is no such left-special word, then after deleting a suitable prefix of $\bfx$ the 
remaining word has no left-special words, which implies that $\bfx$ is ultimately periodic by \eqref{eq:MH}.
\end{proof}

Let $n$ be a positive integer. 
We define the Rauzy graph $\mathcal G_n = \mathcal G_n(\mathbf x)$ of $\bfx$. 
The set of vertices of the oriented graph 
$\mathcal G_n$ is the set of subwords of $\bfx$ of length $n$. 
There exists an edge from $u$ to $v$ in $\mathcal G_n$ if and only if there exist 
letters $x, y$ such that $ux = yv$ is a subword of $\bfx$. 
In this case, we write 
$$
[uv] = ux = yv. 
$$
A finite sequence of edges $[u_1u_2], [u_2u_3],\dots, [u_{k-1}u_k]$ is called a path and we write it as $[u_1u_2 \dots u_{k-1}u_k]$.
A subword of length $n+\ell$ of the word $\mathbf x$ induces a path of length $\ell$ in $\mathcal G_n$.
Moreover, a path of length $\ell$ in $\mathcal G_{n+k}$ induces a path of length $\ell + k$ in $\mathcal G_{n}$. 
However, the converse does not hold in general.

Let $U$ and $V$ be two paths in $\mathcal G_n$. 
If the last vertex of $U$ and the initial vertex of $V$ are identical, we define the concatenated path $UV$ in a natural way.
We denote the length of a path $U$ in $\cG_n$ by $|U|_n$. 
For example, we have $|U|_n=k$ and $|U|=n+k$, where $| \cdot |$ denotes the 
length (that is, the number of letters) of a word.

In the Rauzy graph, any non-right special vertex has exactly one outgoing edge, and any
non-left special vertex has exactly one incoming edge, unless it is a prefix of $\mathbf x$ having no other occurrence in $\bfx$. 
Any right special vertex has at least two outgoing edges.
Any left special vertex has at least two incoming edges.

\subsection{Recurrent words}

\begin{definition}
A subword $w$ of an infinite word $\mathbf x$ is \emph{recurrent} if $w$ occurs in $\mathbf x$ infinitely many times.
An infinite word $\mathbf x$ is \emph{$n$-recurrent} if every subword of $\mathbf x$ of length $n$ is recurrent. 
An infinite word $\bfx$ is called \emph{recurrent} if every subword of $\bfx$ is recurrent. 
\end{definition}

\begin{definition}
For any $n \ge 1$, let $a_n$ denote the shortest prefix of $\mathbf x$ such that $\mathbf x = a_n \mathbf z_n$ and $\mathbf z_n$ is $n$-recurrent.
We call $a_n$ the \emph{$n$-th non-recurrent prefix} of $\mathbf x$
and $\mathcal G'_n(\mathbf x) := \mathcal G_n(\mathbf z_{n+1})$ the \emph{$n$-th reduced Rauzy graph} of $\mathbf x$,
which is a subgraph of $\mathcal G_n(\mathbf x)$ consisting of recurrent vertices ($n$-subwords) and recurrent edges ($n+1$-subwords).
We note that the reduced Rauzy graph is strongly connected.
We let $s_n = |a_n|$ denote the length of $a_n$. 
When there is no ambiguity, we simply write $a$ and $\mathbf z$ instead of $a_n$ and $\mathbf z_n$. 
\end{definition}

\begin{lemma}\label{prefix1}
Suppose that $\mathbf z_n \ne \mathbf z_{n+1}$.  
Then there exists a bi-special $(n-1)$-word $w$ in $\mathbf z_n$ such that $x w y$ is an $(n+1)$-subword of $\mathbf z_{n}$ but $x w y$ is not a subword of $\mathbf z_{n+1}$, where $w y$ is the prefix of length $n$ of $\mathbf z_{n+1}$ and $x$ is the last letter of $a_{n+1}$. 
In particular, $wy$ 
is a left-special word.
\end{lemma}

\begin{proof}
We write $\bfx = x_1 x_2 \dots$. 
Let $\bfx = [u_1 u_2 \dots ]$ with $u_i = x_i x_{i+1} \dots x_{i+n-1}$. 
Let $s := s_{n+1} = |a_{n+1}|$. 
Then $\mathbf z_{n+1} = x_{s+1} x_{s+2} \dots = [u_{s+1} u_{s+2}  \dots ]$ and $[u_s u_{s+1}]$ is not recurrent.
Since both $u_s$ and $u_{s+1}$ are both recurrent, $u_s$ is right special and $u_{s+1}$ is left special.
Let $u_s =xw$ and $u_{s+1} = wy$. 
Then $w$ is a bi-special $(n-1)$-word by Lemma~\ref{special}. 
Therefore, there are letters $x'$ and $y'$ such that $x' \ne x$, $y' \ne y$ and the $n$-words $x'w$ and $wy'$ are factors of $\mathbf z_{n}$.
However $xwy = [u_s u_{s+1}] = x_s \dots x_{s+n}$ is not recurrent, thus it is not a subword of $\mathbf z_{n+1}$.
\end{proof}

\begin{lemma}\label{prefix2}
Let $n \ge 1$. Then we have 
\begin{equation*}
p(n + s_n -1,\bfx) = p(n + s_n -1,\bfz_{n})+ s_n.
\end{equation*}
\end{lemma}

\begin{proof}
If $s_n = 0$, then $\bfx = \bfz_0$, thus the lemma holds. 
We assume that $s := s_n = |a_n| \ge 1$. 
We write $\bfx = x_1 x_2 \dots = [u_1 u_2 \dots ]$ 
with $u_i = x_i x_{i+1} \dots x_{i+n-1}$. 
Then $u_s$ is not recurrent and $u_{i}$ is recurrent for $i \ge s+1$.

For each $i$ with $1 \le i \le s$, the subword
$[u_i \dots u_{i+s-1}]$  contains $u_{s}$.
However, for $i \ge s+1$, the subword $[u_i \dots u_{i+s-1}]$ is a factor of $\bfz_n = [u_{s+1} u_{s+2} \dots ]$ and do not contain any non-recurrent $n$-subword.

Since $u_j \ne u_s$ for $j \ge s+1$,
we deduce that the $s$ words $[u_i \dots u_{i+s-1}]$ 
with $i=1, \ldots , s$ are all distinct.
Therefore, we have
$p(n + s -1,\bfx) = p(n + s -1,\bfz_{n})+ s.$
\end{proof}

The following lemma is \cite{CaNi10}*{Theorem 4.5.4}.

\begin{lemma}\label{xp}
Let $\bfx = x_1 x_2 \dots $ be a non-ultimately periodic word. 
Let $n \ge 1$. 
Suppose that $p(n+1,\mathbf x) = p(n,\mathbf x)+1$. 
Then there is one right-special word of length $n$. 
Moreover, if $x_1 \dots x_{n}$ occurs in $\bfx$ more than once, then there 
is a unique left-special word of length $n$ in $\bfx$ and $\bfx$ is $n$-recurrent. 
\end{lemma}

\begin{definition}
A right-special word $w$ is called \emph{essential} 
if $w$ is a suffix of a right-special word of any length 
(or of arbitrarily large length, this is equivalent).
\end{definition}

From Lemma~\ref{special}, we note that for any non-essential right-special word $w$ there exists $m$ such that $w$ is not a suffix of a right special word of length bigger than $m$.
Therefore, if $\bfx$ is not ultimately periodic, then for each $n$ there exists an essential right-special word of length $n$.

\begin{lemma}\label{especial}
Suppose that $\liminf_{n\to\infty}p(n,\bfx)/n<2$ and $\bfx$ is not ultimately periodic.
Then, for each $n$, there exists a unique essential right-special word of length~$n$.
\end{lemma}

\begin{proof}
By the above remark, there exists an essential right-special word of length $n$.
By assumption, there is an integer $m > n$ such that $p(m+1,\bfx) = p(m,\bfx) +1$. 
By Lemma~\ref{xp}, there exists one right-special word of length $m$. 
Therefore, the number of essential right-special words of length $n$ is at most one.
\end{proof}

\begin{lemma}\label{nn+1}
Suppose that $\bfx$ is $n$-recurrent and non-ultimately periodic.
If $p(n+1,\mathbf x) = p(n,\mathbf x) + 1$, then $\bfx$ is $(n+1)$-recurrent.  
\end{lemma}

\begin{proof}
Write $\bfx = a \mathbf z$, where $a$ is the $(n+1)$-th non-recurrent prefix of $\bfx$ and assume that $a$ is non-empty.
Let $\mathbf z = x_{s+1} x_{s+2} \dots$. 
Then $x_{s+1} \dots x_{s+n}$ is a left-special $n$-word of $\bfx$ by Lemma~\ref{prefix1}. 
Since there is only one left special $n$-word in $\bfx$ by Lemma~\ref{xp} and the only left special word has two preceding letters,
Lemma~\ref{nonep} implies that $x_{s} x_{s+1} \dots x_{s+n}$ occurs infinitely many times in $\bfx$, which contradicts the choice of $a$ and $\mathbf z$.
\end{proof}

\subsection{Rauzy graph of $\infty$-shape}

A Rauzy graph $\mathcal G_n$ (resp., a reduced Rauzy graph $\mathcal G'_n$) is called of \emph{$\infty$-shape} 
if there exist exactly one left special word and one right special word of length $n$ in $\mathbf{x}$ (resp., in $\mathbf z_{n+1}$), both of which coincide with the word $w$,
and there are exactly two paths from $w$ to itself.
Note that $w$ is a bi-special word.

\begin{lemma}\label{pn+1}
Suppose that a non-ultimately periodic word $\mathbf x$ satisfies that $p(n+1,\mathbf x) = p(n,\mathbf x) +1$ for some positive integer $n$.
Then the reduced Rauzy graph $\mathcal G'_n(\mathbf x)$ is isomorphic to one of the graphs
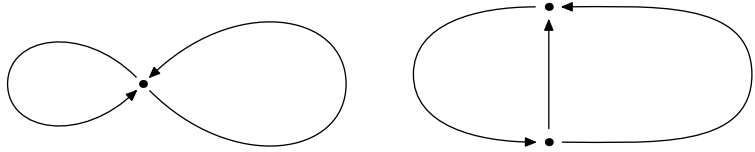
\begin{center}
\includegraphics{Rauzy_graph-4.mps}
\qquad
\includegraphics{Rauzy_graph-5.mps}
\end{center}
Moreover, in the latter case,  
there exists a positive integer $d$ depending on $n$ such that $\mathcal G'_{n+d}(\mathbf x)$ is of $\infty$-shape.
\end{lemma}

\begin{proof}
Let $\mathbf x = a \mathbf z$ with the $n$-th non-recurrent prefix (possibly empty) $a$.
Then, $p(n+1,\mathbf z) = p(n,\mathbf z)+1$ and $\mathbf z$ has either one right special word and one left special word, or one bispecial word of length $n$. 

If there is a bi-special word $w$ of length $n$ in $\mathbf z$, then there are no other special vertices. 
Let $[w u_1 \cdots u_{k-1} w]$ and $[w v_1 \cdots v_{\ell-1} w]$ denote the 
two cycles containing $w$. 
When $k=1$, then $[w u_1 \cdots u_{k-1} w]$ denotes $[ww]$.
The Rauzy graph $\mathcal G_n(\mathbf z)$ can be represented as follows:
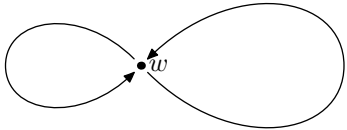
\begin{center}
\includegraphics{Rauzy_graph-9.mps}
\end{center}

If the right-special word and the left special word are distinct, then $\mathcal G_n(\mathbf z)$ is either of the form 
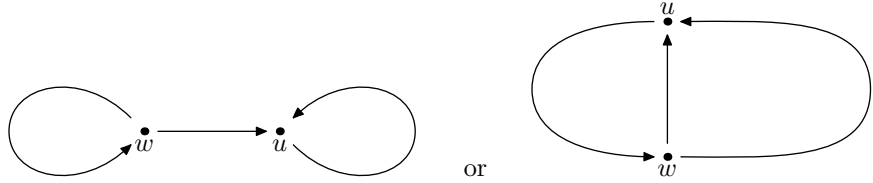
\begin{center}
\includegraphics{Rauzy_graph-3.mps}
\quad \text{ or } \quad
\includegraphics{Rauzy_graph-6.mps}
\end{center}
For the left-hand side graph, $\mathbf z$ is ultimately periodic, which is not allowed.
Let $u$ denote the left special word and $w$ denote the right special word.
In the Rauzy graph $\mathcal G_n(\mathbf z)$, there exist two paths from $w$ to $u$, namely, 
$[w u_1 \cdots u_{k-1} u]$ and $[w v_1 \cdots v_{\ell-1} u]$ and there exists one path from $u$ to $w$, namely $[u w_1 \dots w_{m-1} w]$.
Then, the $(n+m)$-subword $[u w_1 \dots w_{m-1} w]$ is bispecial since $[u w_1 \dots w_{m-1} w u_1]$, $[u w_1 \dots w_{m-1} w v_1]$ 
and  $[u_{k-1} u w_1 \dots w_{m-1} w]$, $[v_{\ell-1} u w_1 \dots w_{m-1} w]$ are edges of the graph $\mathcal G_{n+m}(\mathbf z)$.
The two cycles containing the bispecial word are  
$$
[u w_1 \dots w_{m-1} w u_1 u_2 \cdots u_{k-1} u w_1 \dots w_{m-1} w]
$$
and 
$$
[u w_1 \dots w_{m-1} w v_1 v_2 \cdots v_{\ell-1} u w_1 \dots w_{m-1} w].
$$
Consequently, the graph $\mathcal G_{n+m}(\mathbf z)$ is of $\infty$-shape. 

Any right special word of $\bfz$ of length $n+d$ for $0 \le d \le m$ is a suffix of 
$[u w_1 w_2 \dots w_{m-1} w]$.
Therefore, $p(n+d+1,\mathbf z) = p(n+d,\mathbf z)+1$, for $0 \le d \le m$.
By Lemma~\ref{nn+1}, the word $\mathbf z$ is $(n+m+1)$-recurrent.
Hence, we have $\mathcal G'_{n+m}(\bfx) = \mathcal G_{n+m}(\mathbf z)$.
\end{proof}

\section{Evolution of the Rauzy graphs}

In this section, we consider the evolution of the Rauzy graph of $\mathbf x = x_1 x_2 \ldots$ when $\bfx$ is not ultimately periodic and 
$p(n,\mathbf x) \le \frac 43 n$ for every large integer $n$.
Then Lemma~\ref{pn+1} implies that there exist infinitely many $n$ such that the reduced Rauzy graphs $\cG'_n(\bfx)$ are of $\infty$-shape. 
Let the Rauzy graph $\cG_n(\mathbf x)$ be of $\infty$-shape with the bispecial word $w$, and let
$U = [w u_1 \dots u_{k-1} w]$, $V = [w v_1 \dots v_{\ell-1} w]$ denote the two cycles of $\cG_n(\bfx)$ from $w$ to itself 
and put $u_0 = u_k = v_0  = v_\ell = w$.
See Figure~\ref{fig1}.
Note that if $k=1$, then $U=[ww]$.  
We write concatenated paths in the Rauzy graph $\cG_n$ as 
$UU = [w u_1 \dots u_{k-1}w u_1 \dots u_{k-1}w]$, 
$UV= [w u_1 \dots u_{k-1}w v_1 \dots v_{\ell-1}w]$
and $V^b = V \cdots V$ 
is called the $b$-th concatenated path of $V$.
We also define 
$$
U_i = [u_i \dots u_{k-1} w], \quad i = 1, \ldots, k-1,
$$ 
and 
$$
V_i = [v_i \dots v_{\ell-1} w], 
\quad i = 1, \ldots, \ell-1. 
$$
We write $U_0 = U$, $V_0 = V$ and for convenience we let $U_k = V_\ell$ 
denote the empty path.

By Lemma~\ref{especial}, either $[u_{k-1}w]$ or $[v_{\ell-1}w]$ is essentially right-special.
If $[u_{k-1}w]$ is essentially right-special, then so is $U$.

\begin{figure}
\begin{center}
\includegraphics{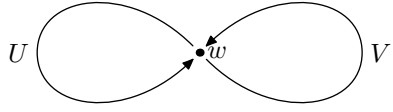}
\end{center}
\caption{The Rauzy graph $\cG_n(\bfx)$ with a bispecial word $w$ and two cycles $U$, $V$}\label{fig1}
\end{figure}

\begin{lemma}\label{lem:figure8}
Suppose that $\bfx$ is not ultimately periodic
and the Rauzy graph $\cG_n(\bfx)$ is of $\infty$-shape with the bi-special word $w$ and two cycles $U = [w u_1 \dots u_{k-1} w]$, $V = [w v_1 \dots v_{\ell-1} w]$.
We assume that, for any $m > n$, we have
$$
\frac{p(m,\mathbf x)}{m} \leq \frac{4}{3}. 
$$
Assume further that $U$ is essentially right-special.

\noindent
(i)  
If $U V^{b} U$ is a factor of $\bfx$ for some integer $b \ge 2$, then
$$ 
k \ge (2b-3)\ell +4.
$$

\noindent
(ii) 
If  $UV^bU$ and $U V^{b'} U$ are factors of $\bfx$ for some integers $b, b' \ge 1$, then $b = b'$.
\end{lemma}

\begin{proof}
(i)
If $UV^bU$ is a factor of $\bfx$,
then $V^{b-1}$ is a right-special word.
Let $m$ be an integer with $n +1 \le m \le n + (b-1)\ell$.
Then by Lemma~\ref{special}, the suffix of $V^{b-1}$ of length $m$ is a right-special word. 
Since $U$ is essentially right-special, there is another right-special word ending with $[u_{k-1}w]$ of length $m$. 
Therefore, $p(m+1,\bfx) \ge p(m,\bfx) +2$ for each $n+1 \le m \le n+(b-1)\ell$.
Thus
$$p(n+(b-1)\ell+1,\bfx) \ge p(n+1,\bfx) + 2 (b-1) \ell.$$
Since $p(n+1,\bfx) = k+ \ell \ge n+2$,
we obtain 
$$
\frac{4}{3} \ge \frac{p(n+(b-1)\ell+1,\mathbf x)}{n+(b-1) \ell+1} \ge \frac{k+(2b-1)\ell}{k+b\ell-1} = 1+ \frac{(b-1)\ell+1}{k+b\ell-1}.
$$
Therefore, we get
$$
k \ge (2b-3)\ell+4. 
$$

(ii) We may assume that $b > b'$.
Note that $V^{b-1}$ and $UV^{b'}$ are right-special and 
by Lemma~\ref{special} the suffixes of them are right-special. 
The number of suffixes of $V^{b-1}$
ending by $[v_{\ell-1}w]$ is $(b-1)\ell$ and 
the number of suffixes of $UV^{b'}$
ending by $[v_{\ell-1}w]$ is $k + b'\ell$.
Since $b'\ell$ of these suffixes coincide, 
the number of right-special subwords of $\bfx$ ending by $[v_{\ell-1}w]$ 
is at least
\[
 (b-1)\ell+k+b'\ell-b'\ell=k+(b-1)\ell,
\] and so
\begin{align*}
    p(n+k+(b-1)\ell+1,\mathbf x) &\ge p(n+1,\bfx) +2k+ 2(b-1)\ell \\
    &=3k+(2b-1)\ell.
\end{align*}
Thus, since $b\ge 2$, we have
$$
\frac{p(n+k+(b-1)\ell+1,\mathbf x)}{n+k+(b-1)\ell+1} \ge \frac{3k+(2b-1)\ell}{2k+b\ell-1} 
> \frac{4}{3},
$$
which gives a contradiction.
\end{proof}

\begin{definition}
Suppose that the Rauzy graph $\cG_n(\bfx)$ consists of one bispecial word $w$ and two cycles $U$, $V$. 
We call $U$ a special cycle if $U$ is essentially right-special.
A Rauzy graph $\cG_n(\bfx)$ with a bispecial word $w$, a special cycle $U$, and a non-special cycle $V$ is called a Rauzy graph of configuration $(w, U, V)$.
We say that the Rauzy graph $\mathcal{G}_{n} (\bfx)$ of configuration $(w,U,V)$ has multiplicity $b \ge 1$ 
if the concatenated path $UV^bU$ is a factor of $\bfx$.
\end{definition}

We note that for Sturmian words, a \emph{standard pair} (cf. \cite{Loth02}*{Section 2.2.1}) $(U, V)$  is mapped to $(U,UV)$ or $(VU,V)$.
It corresponds to the case $b =1$ in the evolution of the Rauzy graph, after the Rauzy graph of configuration $(w, U,V)$, the next graph of $\infty$-shape has configuration $(U, UU, UVU)$. 
If $b = 1$ for all large $n$ with $\mathcal G_n(\bfx)$ of $\infty$-shape, 
then the word $\bfx$ is quasi-Sturmian.

\begin{proposition}\label{bge1}
Suppose that the reduced Rauzy graph $\cG'_n(\bfx)$ is of configuration $(w, U, V)$ with multiplicity $b \ge 1$ and let $|U|_n = k$ and $|V|_n = \ell$. 
Assume that 
$$
\frac{p(m,\mathbf x)}{m} \leq \frac 43 \quad \text{ for } m \ge n.
$$
Then the reduced Rauzy graph $\mathcal G'_{n+k}(\bfx)$ is of configuration either $(U, UU, UV^bU)$
or $(U, UV^bU, UU)$
with the bispecial word $U$
and the two cycles of concatenated paths $UU$ and $UV^bU$ (see Figure~\ref{fig2}).
Furthermore, we have   
\begin{equation}\label{pboundslope}
p(n+(b-1)\ell+1,\bfz_{n+k+1}) = 
k + (2b-1)\ell 
\end{equation} 
and
\begin{equation}\label{sjump}
s_{n+(b-1)\ell+2}=s_{n+k+1}.
\end{equation}
Moreover,
if $k < \ell$ and $s_{n+1} < s_{n+k+1}$
with $b = 1$, then 
there exists $i$, $1 \le i \le \ell-1$ such that 
\begin{equation}\label{bddcycle}
\bfz_{n+1} = V_i V U \dots, \ \mathbf z_{n+k+1} = V U \dots ,\ \text{ and } \ s_{n+k+1} - s_{n+1} = |V_i|_n,
\end{equation}
\end{proposition}

\begin{figure}
\begin{center}
\includegraphics{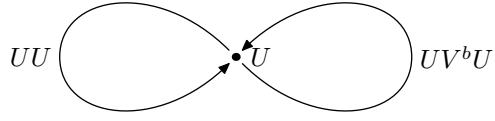}
\end{center}
\caption{The reduced Rauzy graph $\mathcal G'_{n+k}(\bfx)$}\label{fig2}
\end{figure}

\begin{proof}
Since $\bfx$ is not ultimately periodic, both $U$ and $V$ are recurrent.
By Lemma~\ref{lem:figure8} (ii), $UV^bU$ is recurrent as well.
Since $U$ is an essential right-special word, $UU$ is also recurrent.
By Lemma~\ref{lem:figure8}, we note that for any $b \ge 1$ we have 
\begin{equation*}
k \ge (b-1)\ell +1.
\end{equation*}
All the recurrent words of length $n+(b-1)\ell+1$ are prefixes of
$$
\begin{cases}
U_i, &\text{ for } \ 0 \le i \le k-(b-1)\ell-1, \\
U_i U, U_i V^b U, &\text{ for } \ k-(b-1)\ell \le i \le k-1, \\
V_i U, V_i V U, \dots, V_i V^{b-1}U, &\text{ for } \ 0 \le i \le \ell-1.
\end{cases}
$$
Using Lemma~\ref{lem:figure8} (ii), we deduce that there is no other recurrent word of length $n+(b-1)\ell+1$. Consequently, 
\begin{equation}\label{pbl}
\begin{split}
p(n+(b-1)\ell+1,\mathbf z_{n+(b-1)\ell+1})
&= p(n+(b-1)\ell+1,\mathbf z_{n+k+1}) \\
&= k + (2b-1)\ell.
\end{split}
\end{equation}

Suppose $\bfz_{n+(b-1)\ell+2} \ne \bfz_{n+k+1}$.
Then by Lemma~\ref{prefix1}, there exists a recurrent bispecial word $W$ of length $m$ with $n + (b-1)\ell +1 \le m \le n+k-1$.
By Lemma~\ref{special}, the bispecial word $W$ should be a path in $\cG_n$ that starts and ends with $U$ or $V$.
However, since $|U| = n+k$, the word $W$ cannot contain $U$.
Thus $W = V^c$ for some $c > b$.
Since $W$ is recurrent, $UV^{c'}U$ for $c' \ge c$ is a factor of $\bfz_n$,
which is not allowed by Lemma~\ref{lem:figure8} (ii).
Hence, we have $s_{n+(b-1)\ell+2} = s_{n+k+1}$.

We assume that $b = 1$ and $s_{n+1} < s_{n+2} = s_{n+k+1}$.
Since $U$ is the   
special cycle, $UU$, $UV$ are recurrent factors. 
We also check that $VU$ is recurrent, otherwise $\bfx$ is ultimately periodic.
Therefore, by Lemma~\ref{prefix1},
$[v_{\ell-1} w v_1]$ is not a recurrent factor of $\bfx$ and
\begin{equation}\label{b1}
x_{s} x_{s+1} \dots = [ v_{\ell-1} w v_1 \dots ], \qquad s: = s_{n+2}.
\end{equation}
Suppose that $k < \ell$.
We apply Lemma~\ref{prefix2} for $n+2$.
Then we have 
$$
p(n+s_{n+2}+1,\bfx) = p(n+s_{n+2}+1,\bfz_{n+2}) + s_{n+2}
\ge n + 2s_{n+2} +2.
$$
Combined with $p(n+s_{n+2}+1,\bfx) \le \frac 43 (n+s_{n+2}+1)$ and $p(n,\mathbf z_{n+1}) = k + \ell - 1 \ge n +1$, we have  
$$
2s_{n+2} + 2 \le n \le k + \ell -2 < 2\ell -2.
$$
Therefore, $s_{n+2} < \ell -2$ and \eqref{b1} yields that
$$
\bfz_{n+1} = [v_i \dots v_{\ell-1} w v_1 \dots ], \quad \text{ for some $i$ with } \ 2 < i \le \ell-1.   
$$
Hence 
\[
\bfz_{n+1} = V_i V U \dots, \quad \bfz_{n+2} = VU \dots, \ \text{ and } \ s_{n+2} - s_{n+1} = |V_i|_n. \qedhere
\]
\end{proof}

\begin{definition}\label{def_ni}
Suppose that $\bfx$ is not ultimately periodic and $p(n,\mathbf{x}) \le \frac {4n}3$ for every large integer $n$.
By Proposition \ref{bge1}, we have an increasing sequence of integers $(n_i)_{i=0}^\infty$ such that 
the reduced Rauzy graph $\mathcal G'_{n_i} (\mathbf x)$ is of $\infty$-shape. 
We will write $\mathcal G'_{n_{i}} (\mathbf x)$ is of configuration $(w_{(i)},U_{(i)},V_{(i)})$
with multiplicity $b_i$ and $k_i = |U_{(i)}|_{n_i}$, $\ell_i = |V_{(i)}|_{n_i}$.
We note that $n_i = | w_{(i)} |$.
\end{definition}

\begin{definition}\label{def_bc}
We say that $\mathbf x$ has \emph{bounded cycles} if there exists $M$ such that all the reduced Rauzy graphs $\mathcal G'_{n}(\bfx)$ have a cycle whose length is bounded by $M$. 
\end{definition}

\begin{proposition}\label{prop:cycles}
Suppose that $\bfx$ is not ultimately periodic and $p(n,\mathbf{x}) \le \frac {4n}3$ for every large integer $n$.

(i) If $\mathbf x$ has bounded cycles, then 
$k_i$ is constant and $b_i =1$ for all large $i$.

(ii) If $\mathbf x$ does not have bounded cycles, then 
there are infinitely many $n_i$'s such that 
$k_i > \ell_i$.

In both cases, we have 
$\lim_{i\to\infty} \ell_i =\infty.$
\end{proposition}

\begin{proof}
By Proposition~\ref{bge1}, we have 
$$n_{i+1} = |w_{(i+1)}| = | U_{(i)} | = | U_{(i)} |_{n_i} +n_i = k_i + n_i$$ and the reduced Rauzy graph $\mathcal G'_{n_{i+1}}$ is of configuration
$(U_{(i)}, U_{(i)}U_{(i)}, U_{(i)} (V_{(i)})^{b_i}U_{(i)})$
or 
$(U_{(i)}, U_{(i)}(V_{(i)})^{b_i}U_{(i)}, U_{(i)}U_{(i)})$.
If $U_{(i+1)} = U_{(i)}U_{(i)}$, $V_{(i+1)} = U_{(i)}(V_{(i)})^{b_i} U_{(i)}$
then 
\begin{align*}
k_{i+1} &= |U_{(i)}U_{(i)}|_{n_{i+1}}
= |U_{(i)}U_{(i)}|_{n_{i}} - k_i = 2k_i - k_i = k_i, \\
\ell_{i+1} &= |U_{(i)}(V_{(i)})^{b_i}U_{(i)}|_{n_{i+1}}
= |U_{(i)}(V_{(i)})^{b_i}U_{(i)}|_{n_{i}} - k_i
= 2k_i + b_i \ell_i - k_i
=k_i + b_i\ell_i.
\end{align*}
If $U_{(i+1)} = U_{(i)}(V_{(i)})^{b_i}U_{(i)}$, $V_{(i+1)} = U_{(i)}$, then 
$$k_{i+1} = |U_{(i)} (V_{(i)})^{b_i}U_{(i)}|_{n_{i+1}} = k_i + b_i\ell_i, \quad
\ell_{i+1} = |U_{(i)}U_{(i)}|_{n_{i+1}} = k_i.$$
If $k_{i+1} = k_i$ for all large $i$'s, then $k_i$ is bounded and $\ell_i$ goes to infinity.
By Lemma~\ref{lem:figure8}, $b_i=1$ for all large $i$'s.

If $k_{i+1} > k_i$ for infinitely many $i$'s, then $k_{i+1} = k_i + b_i\ell_i > k_i = \ell_{i+1}$ for infinitely many $i$'s.
Note that $(k_i+\ell_i)_{i=0}^{\infty}$ is strictly increasing. 
We see that $k_i$ goes to infinity and so does $\ell_i$.
\end{proof}

\begin{example}
Let 
$\mathbf{x}=U^{a_1}V U^{a_2}V \cdots VU^{a_n}V \cdots$ with $a_1 \le a_2 \le \dots $ for some cycles $U = [w u_1 \dots u_{k-1} w]$ and $V = [w v_1 \dots v_{\ell-1} w]$ with $n$-subwords $w$, $u_j$ ($1\le j \le k-1$) and $v_j$ ($1\le j \le \ell-1$).
Let $n_i$ ($n_i > n$) be the sequence such that the reduced Rauzy graph $\mathcal G'_{n_i} (\bfx)$ is of $\infty$-shape. Then $n_i = n + ik$ and $\mathcal G'_{n_i} (\bfx)$ has cycles $U^{i+1}$ and $U^{i} V U^{i}$.
Therefore, $\bfx$ has bounded cycles. 
\end{example}

\begin{example}
Let $w_0 = 0$, $w_1 = 00011$.
We define recursively 
$$w_{n+1} = w_n w_n w_n w_{n-1} w_{n-1} \quad \text{ for } n \ge 1.$$
We define $\mathbf x = \lim_{n\to\infty} w_n$.
Then $\mathbf x$ is recurrent and $\mathcal G'_{n}(\mathbf x) = \mathcal G_{n}(\mathbf x)$ for all $n$.
Let $n_0 = 3$. 
Then 
$w_{(0)} =000$, $U := U_{(0)} = 00011000$, $V := V_{(0)} =0000$ with multiplicity $b_0=2$. 
Therefore, $k_0 = |00011000|_3 =5$, $\ell_0 = |0000|_3 = 1$. 
In what follows, we define the concatenation of $U$ and $V$ as paths of $\mathcal G_3$. 
In particular, $UU= 0001100011000$ and 
$UVVU = 000110000011000$.
We check that
\begin{align*}
w_{(1)} &= U, & U_{(1)} &= UU, &V_{(1)} &= UVVU, &b_1 &= 1, \\ 
w_{(2)} &= UU, & U_{(2)} &= UUU, &V_{(2)} &= UUVVUU, &b_2 &= 1, \\
w_{(3)} &= UUU, & U_{(3)} &= UUUVVUUU, &V_{(3)} &= UUUU, &b_3 &= 2, \\
w_{(4)} &= U^3V^2U^3, & U_{(4)} &= U^3V^2U^3V^2U^3 , &V_{(4)} &= U^3V^2U^5V^2U^3, &b_4 &= 1
\end{align*}
and 
\begin{align*}
n_1 
&= n_0+k_0 = 8, &k_1 &= k_0 = 5, &\ell_1 &= k_0 + b_0 \ell_0 = 7, \\
n_2 &= n_1+k_1 = 13, &k_2 &= k_1 = 5, &\ell_2 &= k_1 + b_1 \ell_1 = 12, \\
n_3 &= n_2+k_2 = 18, &k_3 &= k_2 + b_2\ell_2 = 17, &\ell_3 &= k_2 = 5, \\
n_4 &= n_3+k_3 = 35, &k_4 &= k_3 = 17, &\ell_4 &= k_3 + b_3\ell_3 = 27.
\end{align*}
\end{example}

\section{Recurrence time of low complexity sequences}

Recall that $s_n$ denotes the length of $n$-th non-recurrent prefix of $\bfx$. Moreover, we use the same notation for $U = [w u_1 \dots u_{k-1} w]$, $V = [w v_1 \dots v_{\ell-1} w]$ as in the previous section.
Throughout this section, $\rho$ is a positive real number less than $4/3$ 
and we assume that $\bfx$ is not ultimately periodic and that 
$p(m,\bfx)/m < \rho$ for sufficiently large $m$.

\begin{lemma}\label{skln}
Let $n$ be given and 
suppose that 
$$
\frac{p(m,\mathbf x)}{m} < \rho
$$
holds for any $m\geq n$.
Suppose that $\mathcal G'_n (\bfx)$ is of $\infty$-shape with two cycles $U$ and $V$. 
Let $b$ be the multiplicity of $\mathcal G'_n (\bfx)$.
Let 
$|U|_n = k$ and $|V|_n = \ell$.
Then we have 
\begin{equation}\label{rec_bound1}
k +  \frac{2b+1} 3\ell < \rho (n + 1)
\end{equation}
and
\begin{equation}\label{rec_bound2}
2s_{n+k+1} + k + \left( 2b-1 \right) \ell  
< \frac{\rho n}{2-\rho}.
\end{equation}
\end{lemma}

\begin{proof}
By Proposition~\ref{bge1}, we have $s := s_{n+(b-1)\ell+2} = s_{n+k+1}$. 
We apply Lemma~\ref{prefix2} with $n$ replaced by $n+(b-1)\ell+2$. Then, 
combined with \eqref{pboundslope}, we have 
\begin{equation*}
\begin{split}
p(n+(b-1)\ell + s+1, \bfx)
&= p(n+(b-1)\ell + s+1,\bfz_{n+(b-1)\ell+2}) + s \\
&\geq p(n +(b-1)\ell+1, \bfz_{n+(b-1)\ell+2})+ 2s\\
&=k+(2b-1)\ell+2s.
\end{split}
\end{equation*}
By using $p(n +(b-1)\ell+ s +1, \bfx) < \rho(n +(b-1)\ell+ s +1)$, we deduce that 
\begin{equation}\label{ineq}
(2-\rho) s + k + ( 2b-1 -\rho(b-1) )\ell < \rho (n+1).
\end{equation}
Therefore, we have
\begin{equation*}
k + \left( \frac{2b}3 + \frac 13 \right)\ell \leq k + \left( 2b-1-\rho (b-1) \right)\ell < \rho (n + 1),
\end{equation*}
which is \eqref{rec_bound1}.

We note that, since $p(n,\mathbf z_{n+1}) = k + \ell - 1 \ge n +1$, we have 
\begin{equation}\label{klbound}
n \le k + \ell - 2.
\end{equation}
Combined with \eqref{ineq}, this gives 
\begin{equation*}
(2-\rho) s + k + \ell + (2-\rho)(b-1)\ell < 
\frac{\rho}{2} n + \frac{\rho}{2} (k+ \ell-2) + \rho,
\end{equation*}
that is, 
\begin{equation*}
(2-\rho) s + 
\left(1-\frac{\rho}{2}\right) (k + \ell) + (2-\rho)(b-1)\ell < \frac{\rho}{2} n. 
\end{equation*}
By dividing by $1-\frac{\rho}{2}$, we obtain \eqref{rec_bound2}.
\end{proof}

In the rest of the paper,  we put
$$\delta := \frac{4-3\rho}{2(1+2\rho)(2-\rho)}.$$

\begin{lemma}\label{ubbounded}
We use the same notation as in Lemma \ref{skln}.
Let $n$ be large enough. 
Then we have 
\begin{equation}\label{deltabound1}
s_{n+k+1} + 2k + b\ell < (1- \delta)(n + 2k + b\ell).
\end{equation}
For $k' = 0, \ldots , k-1$,
we have either 
\begin{equation}\label{deltabound22}
s_{n+k+1} + 2k + b\ell
< (1- \delta)(n + k + b\ell + k')
\end{equation}
or
\begin{equation}\label{deltabound21}
s_{n+k+1} + 2k + b\ell + k'
< (1- \delta)(n + 2k).
\end{equation}
Moreover, if $\ell < k$, then 
\begin{equation}\label{deltabound3}
s_{n+k+1} + 2k + b\ell
< (1- \delta)(n + 2k).
\end{equation}
\end{lemma}

\begin{proof}
By \eqref{rec_bound1} and the last statement of Proposition~\ref{prop:cycles}, for large $n$ such that $\ell \ge 3 > 2 \rho$, we have 
\begin{equation}\label{extrabound}
2k + b\ell < 2 \rho n .
\end{equation}
Then we have by \eqref{rec_bound2}
\begin{align*}
s_{n+k+1} + \delta (n + 2k + b\ell) < 
\frac{\rho n}{2(2-\rho)} + (1+2\rho) \delta n
= \frac{\rho n}{2(2-\rho)} + \frac{4-3\rho}{2(2-\rho)} n = n.
\end{align*}
Therefore, we deduce \eqref{deltabound1}.

From \eqref{rec_bound2}, for $k' = 0, \ldots , k-1$,
we have either 
\begin{equation}\label{condition1}
s_{n+k+1} + k < \frac{\rho n}{2(2-\rho)} + k'
\end{equation}
or
\begin{equation}\label{condition2}
s_{n+k+1} + b\ell + k' < \frac{\rho n}{2(2-\rho)}.
\end{equation}
If \eqref{condition1} holds, then by \eqref{extrabound}
\begin{equation*}
s_{n+k+1} + k + \delta (n + k + b\ell +k') < 
\frac{\rho n}{2(2-\rho)} + k' + (1+2\rho) \delta n 
=n +k',
\end{equation*}
which implies \eqref{deltabound22}.
If \eqref{condition2} holds, then by \eqref{extrabound}
\begin{align*}
s_{n+k+1} + b\ell + k' + \delta (n + 2k) < 
\frac{\rho n}{2(2-\rho)} + (1+2\rho) \delta n
= n.
\end{align*}
which implies \eqref{deltabound21}.

If $\ell < k$, then by \eqref{rec_bound2}
\begin{equation*}
s_{n+k+1}  + b\ell < \frac{\rho n}{2(2-\rho)}.
\end{equation*}
Using \eqref{extrabound}, we have 
\begin{equation*}
s_{n+k+1} + b\ell + \delta (n + 2k) < 
\frac{\rho n}{2(2-\rho)} + (1+2\rho)\delta n
= n. 
\end{equation*}
Therefore, we obtain \eqref{deltabound3}.
\end{proof}


For a sequence $\mathbf{x}$ with bounded cycles, the growth of the subword complexity 
function is due to the growth of the lengths of the initial non-recurrent prefixes
of $\mathbf{x}$.
The case of bounded cycles is considered separately. 

\begin{proposition}\label{prop:bounded}
If $\mathbf x$ has bounded cycles, then for infinitely many integers $n$ we have 
$$
r(n,\mathbf x) 
< (1-\delta)n .
$$
\end{proposition}   

\begin{proof}
We use the notations 
introduced in Definition~\ref{def_ni}. 
Using Proposition~\ref{prop:cycles}, we may choose $(n_i)_{i=0}^{\infty}$ as $k_i = k$ is a constant and $b_i = 1$ for all $i \ge 0$.
Set $(w,U,V):=(w_{(0)},U_{(0)},V_{(0)})$
and $\ell := \ell_0$. 
Then we have $w_{(i)}=U^i$, $U_{(i)}=U^{i+1}$ and $V_{(i)}=U^i V U^i$ for each $i \ge 1$. 
We note that $n_i=|w_{(i)}|=n_0+ki$, $|U_{(i)}|=n_0+(i+1)k$, $|V_{(i)}|=n_0 +\ell+2ki$ and that $|U_{(i)}|_{n_i}=k$, $|V_{(i)}|_{n_i}=\ell+ki$. 
Since $\bfx$ is not ultimately periodic, $s_{n_i+1}<s_{n_{i+1}+1}$ for infinitely many $i$'s.  
Choose $i$ large enough to ensure that 
\begin{equation}\label{ellbound}
(\ell + ki)\delta > 2k.
\end{equation}
By Proposition~\ref{bge1}, we see that 
$s_{n_i+2}=s_{n_{i+1}+1}=s_{n_i+k+1}$ and 
\begin{align}\label{subword0}
\bfz_{n_i+2}=V_{(i)}U_{(i)}\ldots,
\end{align}
where the concatenation rule follows from the word of length $n_i$.
Let
\[
h = h(i) :=\max\{ j \ge 1 \mid s_{n_{i+1}+1}=s_{n_{i+j}+1}\}.
\]
Then $s_{n_{i+h}+1} < s_{n_{i+h+1}+1}$. 
We, again, apply Proposition~\ref{bge1}. 
Then we have 
\begin{equation}\label{subwordj}
\bfz_{n_{i+h}+1} = \widetilde{V_{(i+h)}} V_{(i+h)} U_{(i+h)} \dots,  \quad
\bfz_{n_{i+h}+2} = V_{(i+h)} U_{(i+h)} \dots
\end{equation}
and 
$$
s_{n_{i+h}+2}-s_{n_{i+h}+1}=|\widetilde{V_{(i+h)}}|_{n_{i+h}},
$$
where the concatenation rule follows from the word of length $n_{i+h}$ and $\widetilde{V_{(i+h)}}$ is a suffix of $V_{(i+h)}= (U_{(i)})^h V_{(i)} (U_{(i)})^h$.
Since $\bfz_{n_i+2} = \bfz_{n_{i+h}+1}$, by \eqref{subword0} and \eqref{subwordj}, we have 
$\widetilde{V_{(i+h)}}=V_{(i)} (U_{(i)})^h$.
Therefore, we have
\begin{equation}\label{new_sjump}
s_{n_{i+h}+k+1} - s_{n_i+2} = s_{n_{i+h}+2}-s_{n_{i+h}+1} = |V_{(i)} (U_{(i)})^h |_{n_{i+h}} = |V_{(i)}|_{n_i} = \ell +ki
\end{equation}
and
$$
\mathbf z_{n_{i+h}+1}  
= V_{(i)} (U_{(i)})^h V_{(i)} (U_{(i)})^{h+1} \dots, \quad \mathbf z_{n_{i+h}+2} = (U_{(i)})^h V_{(i)} (U_{(i)})^{i+h+1} \dots.
$$
Therefore 
\begin{equation}\label{rb}
r(n_i+(h-1)k,\bfx) \le s_{n_i+2} + (i+1) k+ \ell.
\end{equation}

Note that the lengths of the two cycles of $\mathcal G_{n_{i+h}}'$ are 
$$
|U_{(i+h)}|_{n_{i+h}} = k, \qquad 
|V_{(i+h)}|_{n_{i+h}} = \ell+(i+h)k.
$$
Thus, we get by \eqref{deltabound1} that 
\begin{equation*}
s_{n_{i+h}+k+1} + 2k + \ell+(i+h)k < (1- \delta)(n_{i+h}+ 2k + \ell+(i+h)k).
\end{equation*}
Therefore, combined with \eqref{new_sjump}, we have 
\begin{equation*}
s_{n_i+2} + (h+2i+2)k + 2\ell < (1- \delta)(n_i +(2h+i+2)k+\ell).
\end{equation*}
Combined with \eqref{rb} and \eqref{ellbound}, 
we have 
\begin{align*}
r(n_i+(h-1)k,\bfx) &\le s_{n_i+2} + (i+1) k+ \ell \\
&< (1- \delta)(n_i+ (2h+i+2)k +\ell) - (i+h+1)k -\ell\\
&= (1- \delta)(n_i+ 2(h+1)k) - (h+1)k -\delta(ki+\ell)\\
&< (1- \delta)(n_i+ 2(h+1)k) - (h+3)k \\
&< (1- \delta)(n_i +(h-1)k).
\end{align*}
Since there are infinitely many such $n_i$'s, the proof is complete. 
\end{proof}

\begin{lemma}\label{cases}
Suppose that $\mathcal G'_n(\bfx)$ is of configuration $(w,U,V)$ with $|U|_n = k$, $|V|_n = \ell$ and multiplicity $b \ge 1$.
Then $\mathbf z_{n+k+1}$ belongs to one of the following cases,
where $U'$ (resp. $V'$) denotes a suffix of $U$ (resp. $V$) with $0 \leq |U'|_n=:k'<k$ (resp. $0 \leq |V'|_n = : \ell' <\ell$)
\begin{enumerate}
\item If $\mathbf z_{n+k+1} = U'UU \dots $, then $r(n+ k+ k', \bfx) \le s_{n+k+1} + k$.
\item If $\mathbf z_{n+k+1} = U'U V^b UUU \dots $, then $r(n+ k', \bfx) \le s_{n+k+1} + k$ and $r(n+ 2k, \bfx) \le s_{n+k+1} + 2k + b\ell + k'$.
\item If $\mathbf z_{n+k+1} = U'U V^b UU V^b U \dots $, then $r(n+ 2k+ b\ell+k', \bfx) \le s_{n+k+1} + 2k + b\ell$. 
\item If $\mathbf z_{n+k+1} = U'U V^b U V^b U \dots $, then $r(n+ 2k+ b\ell, \bfx) \le s_{n+k+1} + k + b\ell + k'$. 
\item If $\mathbf z_{n+k+1} = U' V^b UUU \dots $, then $r(n+ 2k, \bfx) \le s_{n+k+1} + k + b\ell + k'$. 
\item If $\mathbf z_{n+k+1} = U' V^b UU V^b U \dots $, then $r(n+ (b-1)\ell, \bfx) \le s_{n+k+1} + \ell + k'$ and $r(n+ k+ b\ell+ k', \bfx) \le s_{n+k+1} + 2k + b\ell$.

\item If $\mathbf z_{n+k+1} = U' V^b U V^b U \dots $, then $r(n+ k+ b\ell+ k', \bfx) \le s_{n+k+1} + k + b\ell$. 
\item If $\mathbf z_{n+k+1} = V' V^c UU \dots $ for some $0 \le c < b$, then $r(n+ k, \bfx) \le s_{n+k+1} + k + c\ell + \ell'$.

\item If $\mathbf z_{n+k+1} = V' V^c U V^b U \dots $ for some $0 \le c < b$, then $r(n+ k+ c\ell+ \ell', \bfx) \le s_{n+k+1} + k + b\ell$.

\end{enumerate}
\end{lemma}

\begin{proof}
By Proposition~\ref{bge1}, every subword of $\bfz_{n+k+1}$ starting with $U$ is 
a concatenation of the two words $UU$ and $UV^bU$.
Therefore, if $\bfz_{n+k+1}$ starts with $U'$, then
the prefix of $\bfz_{n+k+1}$ is 
either $U' UU$, $U' UV^bU$ or $U' V^bU$.
The second and the third cases are divided into the three subcases respectively: 
$U' UV^bUUU$, $U' UV^bUUV^bU$, $U' UV^bUV^bU$, 
and
$U' V^bUUU$, $U' V^bUUV^bU$, $U'V^bUV^bU$. 

If $\bfz_{n+k+1}$ starts with $V'$, a suffix of $V$, then
the prefix of $\bfz_{n+k+1}$ is either 
$V' V^cUU$ or $V' V^cUV^bU$ for some $c \le b-1$, thus giving the last two cases. 
\end{proof}

\begin{proposition}\label{prop:unbounded}
Let $1 \le \rho < 4/3$ be defined as in Lemma \ref{skln}. Moreover, let
$n$ be large enough that $\ell \ge 3$.
Let $\mathcal G'_n(\bfx)$ be of configuration $(w,U,V)$ with multiplicity $b \ge 1$.
Suppose that $|U|_n = k > \ell =|V|_n$.
Then there exists $m$ with $n \le m < n + 3k + b\ell$ satisfying that
$$
r(m,\bfx) < (1 - \delta ) m.
$$
\end{proposition}

\begin{proof}
By Lemma~\ref{cases}, we have 
\begin{enumerate}
\item If $\mathbf z_{n+k+1} = U'UU \dots $,  
then by \eqref{deltabound1} we have
\begin{align*}
r(n+ k+ k', \bfx) &\le s_{n+k+1} + k <
\left( 1 - \delta \right) (n + k ) 
\leq 
\left( 1 - \delta \right) (n + k + k').
\end{align*}

\item If $\mathbf z_{n+k+1} = U'U V^b UUU \dots $, then
by \eqref{deltabound22} or \eqref{deltabound21},  
we have either
$$r(n+ k', \bfx) \le s_{n+k+1} + k < (1- \delta)(n + k')$$
or 
$$r(n+ 2k, \bfx) \le s_{n+k+1} + 2k + b\ell + k' < (1- \delta)(n + 2k).$$

\item If $\mathbf z_{n+k+1} = U'U V^b UU V^b U \dots $, then by \eqref{deltabound1}
\begin{align*}
r(n+ 2k+ b\ell+k', \bfx) &\le s_{n+k+1} + 2k + b\ell \\
&< (1-\delta)(n+ 2k+ b\ell) \leq (1-\delta)(n+ 2k+ b\ell+k'). 
\end{align*}

\item If $\mathbf z_{n+k+1} = U'U V^b U V^b U \dots $, then by \eqref{deltabound1}
\begin{align*}
r(n+ 2k+ b\ell, \bfx) &\le s_{n+k+1} + k + b\ell + k' \\
&\le s_{n+k+1} + 2k + b\ell < (1- \delta)(n + 2k + b\ell).
\end{align*}

\item If $\mathbf z_{n+k+1} = U' V^b UUU \dots $, then by \eqref{deltabound3}
\begin{align*}
r(n+ 2k, \bfx) \le s_{n+k+1} + k + b\ell + k' < s_{n+k+1} + 2k + b\ell
< (1- \delta)(n + 2k). 
\end{align*}

\item If $\mathbf z_{n+k+1} = U' V^b UU V^b U \dots $, 
then by \eqref{deltabound21} or \eqref{deltabound22},
we have either
$$r(n+ (b-1)\ell, \bfx) \le s_{n+k+1} + \ell + k'
\le (1- \delta)n \leq (1- \delta)(n+ (b-1)\ell)$$
or 
$$r(n+ k+ b\ell+ k', \bfx) \le s_{n+k+1} + 2k + b\ell
< (1- \delta)(n + k + b\ell + k').$$

\item If $\mathbf z_{n+k+1} = U' V^b U V^b U \dots $, 
then by \eqref{deltabound1}
\begin{align*}
r(n+ k+ b\ell+ k', \bfx) &\le s_{n+k+1} + k + b\ell \\
&< (1- \delta)(n + k + b\ell) \leq (1- \delta)(n + k + b\ell +k').
\end{align*}

\item If $\mathbf z_{n+k+1} = V' V^c UU \dots $, 
then by \eqref{deltabound3}
\begin{align*}
r(n+ k, \bfx) \le s_{n+k+1} + k + c\ell + \ell' \le s_{n+k+1} + k + b\ell < (1- \delta)(n + k). 
\end{align*}

\item If $\mathbf z_{n+k+1} = V' V^c U V^b U \dots $, then by \eqref{deltabound3}
\begin{align*}
r(n+ k+ c\ell+ \ell', \bfx) \le s_{n+k+1} + k + b\ell 
< (1- \delta)(n + k) \le (1- \delta)(n+ k+ c\ell+ \ell'). 
\end{align*}

\end{enumerate}
\end{proof}

\begin{proof}[Proof of Theorem~\ref{thm2}]
By Proposition~\ref{prop:cycles}, $\bfx$ is either of bounded cycle with multiplicity 1 for large $n_i$ or unbounded cycle with infinitely many $n_i$'s with $|U_{(i)}| > |V_{(i)}|$.
By Propositions \ref{prop:bounded} and \ref{prop:unbounded}, we 
get the first assertion. 
Set 
$$
\mathrm{rep} (\bfx) = \liminf_{n \to \infty} \frac{r(n,\mathbf x)}{n}.
$$
By \cite{BuKi17}*{Lemma 3.6} (recall that the definition of $\mathrm{rep}$ is different there), we have
$$
\mu \Bigl( \, \sum_{j \ge 1} \, \frac{x_j}{b^j} \, \Bigr)  \ge 
1 + \frac{ 1 }{\mathrm{rep} (\bfx) }.
$$
This implies the second assertion.  
\end{proof}

\section*{Acknowledgment}
The authors are thankful to Julien Cassaigne for introducing them to Aberkane's papers.
The second author's work was supported by the JSPS KAKENHI Grant Number 24K06641.
The third author was supported by the National Research Foundation of Korea (NRF-2018R1A2B6001624, RS-2023-00245719).

\end{document}